\documentclass[a4paper, 12 pt]{article}

\usepackage{amsfonts,amsmath,amssymb,hhline,ifthen,amsthm,graphics,cite,latexsym, caption2}
\usepackage{latexsym}
\usepackage[T2A]{fontenc}
\usepackage[cp1251]{inputenc}
\usepackage[english]{babel}
\usepackage[matrix, arrow, curve]{xy}
\usepackage{graphicx}

\newcounter{q3}
\newcommand{\dsm}[3]{
{\if#20{\if#31{\frac{\partial #1}{\partial y}}\else
          {\frac{\partial^{#3} #1}{\partial y^{#3}}}
        \fi}\else
  {\if#30{\if#21{\frac{\partial #1}{\partial x}}\else
            {\frac{\partial^{#2} #1}{\partial x^{#2}}}
          \fi}\else
    {\setcounter{q3}{#2}\addtocounter{q3}{#3}
    \frac{\partial{\if{1}\arabic{q3}^{}\else^{ \arabic{q3} }\fi}#1}
    {{\if#20\else{\partial x{\if#21\else{^{#2}}\fi}}\fi}
     {\if#30\else{\partial y\if#31\else{^{#3}}\fi}\fi} }}
   \fi}
\fi} }

\newcommand{\dzm}[2]{
{\if#10{\if#21{\frac{\partial}{\partial\overline{\zeta}}}\else
          {\frac{\partial^{#2}}{\partial\overline{\zeta}^{#2}}}
        \fi}\else
  {\if#20{\if#11{\frac{\partial}{\partial\zeta}}\else
            {\frac{\partial^{#1}}{\partial\zeta^{#1}}}
          \fi}\else
    {\setcounter{q3}{#1}\addtocounter{q3}{#2}
    \frac{\partial{\if{1}\arabic{q3}^{}\else^{ \arabic{q3} }\fi}}
    {{\if#10\else{\partial\zeta{\if#21\else{^{#1}}\fi}}\fi}
     {\if#20\else{\partial\overline{\zeta}\if#21\else{^{#2}}\fi}\fi} }}
   \fi}
\fi} }

\newcounter{q2}

\newtheoremstyle{theor}
  {\medskipamount}
  {\medskipamount}
  {\itshape}
  {}
  {\bfseries}
  {.}
  {.5em}
  {}

\newtheorem{definition}{Definition}[section]
\newtheorem{theorem}[definition]{Theorem}
\newtheorem{lemma}[definition]{Lemma}
\newtheorem{proposition}[definition]{Proposition}
\newtheorem{corollary}[definition]{Corollary}

\theoremstyle{definition}
\newtheorem{remark}[definition]{Remark}
\newtheorem{example}[definition]{Example}

\numberwithin{equation}{section}

\makeindex

\newtheoremstyle{remarks}
  {0mm}
  {0mm}
  {\itshape}
  {}
  {\itshape}
  {.}
  {.5em}
  {}
\makeatletter \@addtoreset{equation}{section} \makeatother

\begin{document}

\subsection*{\center BOUNDEDNESS OF PRETANGENT SPACES TO GENERAL METRIC SPACES}\begin{center}\textbf{V. Bilet and O. Dovgoshey } \end{center}
\parshape=5
1cm 13.5cm 1cm 13.5cm 1cm 13.5cm 1cm 13.5cm 1cm 13.5cm \noindent \small {\bf Abstract.}
  Let $(X,d,p)$ be a metric space with a metric $d$ and a marked point $p.$ We define the set of $w$-strongly porous at 0 subsets of $[0,\infty)$ and prove that the distance set $\{d(x,p): x\in X\}$ is $w$-strongly porous at 0 if and only if every pretangent space to $X$ at $p$ is bounded.

\medskip

\parshape=2
1cm 13.5cm 1cm 13.5cm  \noindent \small {\bf Key words:} metric spaces, infinitesimal boundedness in metric spaces, distance set, local strong porosity.

 \bigskip
\textbf{ AMS 2010 Subject Classification: 54E35, 28A10.}


\large\section {Introduction } \hspace*{\parindent} Recent achievements in the metric space theory are closely related to some generalizations of the differentiation. A possible but not the only one initial point to develop the theory of a differentiation in metric spaces is the fact that every separable metric space admits an isometric embedding into the dual space of a separable Banach space. It provides a linear structure, and so a differentiation. This approach leads to a rather complete theory of rectifiable sets and currents on metric spaces \cite{AK1, AK2}. The concept of the upper gradient \cite{H1, H2, Sh}, Cheeger's notion of differentiability for Rademacher's theorem in certain metric measure spaces \cite{Ch}, the metric derivative in the studies of metric space valued functions of bounded variation \cite{A1, A2} and the Lipshitz type approach in \cite{H} are important examples of such generalizations. The generalizations of the differentiability mentioned above give usually nontrivial results only for the assumption that metric spaces have ``sufficiently many'' rectifiable curves.

A new intrinsic notion of differentiability for the mapping between the general metric spaces was produced  in \cite{MD} (see
also \cite{DM}). A basic technical tool in \cite{MD} is a pretangent and tangent spaces to an arbitrary metric space $X$ at a point $p.$ The development of this theory requires the understanding of interrelations between the infinitesimal properties of initial metric space and geometry of pretangent spaces to this initial. The main purpose of the present paper is to search the conditions under which all pretangent spaces to $X$ at a point $p\in X$ are bounded.

For convenience we recall some terminology and results related to pretangent spaces to
general metric spaces.

Let $(X,d,p)$ be a pointed metric space with a metric $d$ and a marked point $p.$ Fix a
sequence $\tilde{r}$ of positive real numbers $r_n$ tending to zero. In what follows
$\tilde{r}$ will be called a \emph{normalizing sequence}. Let us denote by $\tilde{X}$
the set of all sequences of points from X and by $\mathbb N$ the set of positive integer (~=~natural) numbers.
\begin{definition}\label{D1.1} Two sequences $\tilde{x}=\{x_n\}_{n\in \mathbb N}$ and $\tilde{y}=\{y_n\}_{n\in \mathbb
N},$ $\tilde{x}, \tilde{y} \in \tilde{X}$ are mutually stable with respect to
$\tilde{r}=\{r_n\}_{n\in \mathbb N}$ if there is a finite limit
\begin{equation}\label{eq1.2}
\lim_{n\to\infty}\frac{d(x_n,y_n)}{r_n}:=\tilde{d}_{\tilde{r}}(\tilde{x},\tilde{y})=\tilde{d}(\tilde{x},\tilde{y}).\end{equation}\end{definition}
We shall say that a family $\tilde{F}\subseteq\tilde{X}$ is \emph{self-stable} (w.r.t.
$\tilde{r}$) if every two $\tilde{x}, \tilde{y} \in \tilde{F}$ are mutually stable. A
family $\tilde{F}\subseteq\tilde{X}$ is \emph{maximal self-stable} if $\tilde{F}$ is
self-stable and for an arbitrary $\tilde{z}\in \tilde{X}$ either $\tilde{z}\in\tilde{F}$
or there is $\tilde{x}\in\tilde{F}$ such that $\tilde{x}$ and $\tilde{z}$ are not
mutually stable.

The standart application of Zorn's lemma leads to the following
\begin{proposition}\label{Pr1.2}Let $(X,d,p)$ be a pointed metric space. Then for every normalizing sequence $\tilde{r}=\{r_n\}_{n\in \mathbb N}$ there exists a maximal self-stable family $\tilde{X}_{p,\tilde{r}}$ such that $\tilde{p}:=\{p,p,...\}\in\tilde{X}_{p,\tilde{r}}.$
\end{proposition}

Note that the condition $\tilde{p}\in\tilde{X}_{p,\tilde{r}}$ implies the equality
\begin{equation*}\label{eq1.3}\lim_{n\to\infty}d(x_n,p)=0 \end{equation*} for every $\tilde{x}=\{x_n\}_{n\in \mathbb N}\in\tilde{X}_{p,\tilde{r}}.$

Consider a function $\tilde{d}:\tilde{X}_{p,\tilde{r}}\times\tilde{X}_{p,\tilde{r}}\rightarrow\mathbb R$ where
$\tilde{d}(\tilde{x},\tilde{y})=\tilde{d}_{\tilde{r}}(\tilde{x},\tilde{y})$ is defined
by \eqref{eq1.2}. Obviously, $\tilde{d}$ is symmetric and nonnegative. Moreover, the
triangle inequality for $d$ implies
$$\tilde{d}(\tilde{x},\tilde{y})\leq\tilde{d}(\tilde{x},\tilde{z})+\tilde{d}(\tilde{z},\tilde{y})$$
for all $\tilde{x},\tilde{y},\tilde{z}\in\tilde{X}_{p,\tilde{r}}.$ Hence $(\tilde{X}_{p,\tilde{r}},\tilde{d})$
is a pseudometric space.
\begin{definition}\label{D1.3} The pretangent space to the space X (at the point p w.r.t. $\tilde{r}$) is the metric identification of the pseudometric space
$(\tilde{X}_{p,\tilde{r}},\tilde{d}).$\end{definition}

Since the notion of pretangent space is important for the paper, we remind this metric
identification construction.

Define the relation $\sim$ on $\tilde X_{p, \tilde r}$ by $\tilde x\sim \tilde y$ if and only if
$\tilde d(\tilde x, \tilde y)=0.$ Then $\sim$ is an equivalence relation. Let us denote
by $\Omega_{p,\tilde r}^{X}$ the set of equivalence classes in $\tilde{X}_{p,\tilde{r}}$ under the
equivalence relation $\sim.$ It follows from general properties of pseudometric spaces
(see, for example, \cite{Kelley}), that if $\rho$ is defined on $\Omega_{p,\tilde
r}^{X}$ by \begin{equation*} \label{eq1.4}\rho(\alpha,\beta):=\tilde d (\tilde x, \tilde
y)\end{equation*}for $\tilde x\in \alpha$ and $\tilde y\in \beta,$ then $\rho$ is a
well-defined metric on $\Omega_{p,\tilde r}^{X}.$ By definition, the metric
identification of $(\tilde{X}_{p,\tilde{r}}, \tilde d)$ is the metric space $(\Omega_{p,\tilde
r}^{X}, \rho).$

It should be observed that $\Omega_{p,\tilde r}^{X}\ne \varnothing$ because the constant
sequence $\tilde p$ belongs to $\tilde{X}_{p,\tilde{r}}.$ Thus every pretangent space
$\Omega_{p, \tilde r}^{X}$ is a pointed metric space with natural distinguished point
$\pi (\tilde p),$ (see diagram~\eqref{eq1.5} below).

Let $\{n_k\}_{k\in\mathbb N}$ be an infinite strictly increasing sequence of natural
numbers. Let us denote by $\tilde r'$ the subsequence $\{r_{n_k}\}_{k\in \mathbb N}$ of
the normalizing sequence $\tilde r=\{r_n\}_{n\in\mathbb N}$ and let $\tilde
x':=\{x_{n_k}\}_{k\in\mathbb N}$ for every $\tilde x=\{x_n\}_{n\in\mathbb N}\in\tilde
X.$ It is clear that if $\tilde x$ and $\tilde y$ are mutually stable w.r.t. $\tilde r,$
then $\tilde x'$ and $\tilde y'$ are mutually stable w.r.t. $\tilde r'$ and
\begin{equation}\label{eqv}\tilde d_{\tilde r}(\tilde x, \tilde y)=\tilde d_{\tilde r'}(\tilde x', \tilde
y').\end{equation} If $\tilde X_{p,\tilde r}$ is a maximal self-stable (w.r.t. $\tilde
r$) family, then, by Zorn's Lemma, there exists a maximal self-stable (w.r.t. $\tilde
r'$) family $\tilde X_{p,\tilde r'}$ such that $$\{\tilde x':\tilde x \in \tilde
X_{p,\tilde r}\}\subseteq \tilde X_{p,\tilde r'}.$$ Denote by $in_{\tilde r'}$ the map
from $\tilde X_{p,\tilde r}$ to $\tilde X_{p,\tilde r'}$ with $in_{\tilde r'}(\tilde
x)=\tilde x'$ for all $\tilde x\in\tilde X_{p,\tilde r}.$ It follows from \eqref{eqv}
that after metric identifications $in_{\tilde r'}$ passes to an isometric embedding
$em':\Omega_{p,\tilde r}^{X}~\rightarrow~\Omega_{p,\tilde r'}^{X}$ under which the
diagram
\begin{equation} \label{eq1.5}
\begin{array}{ccc}
\tilde X_{p, \tilde r} & \xrightarrow{\ \ \mbox{\emph{in}}_{\tilde r'}\ \ } &
\tilde X_{p, \tilde r^{\prime}} \\
\!\! \!\! \!\! \!\! \! \pi\Bigg\downarrow &  & \! \!\Bigg\downarrow \pi^{\prime}
\\
\Omega_{p, \tilde r}^{X} & \xrightarrow{\ \ \mbox{\emph{em}}'\ \ \ } & \Omega_{p, \tilde
r^{\prime}}^{X}
\end{array}
\end{equation}is commutative. Here $\pi$ and
$\pi'$ are the natural projections, $\pi(\tilde x):=\{\tilde y \in \tilde X_{p,\tilde
r}: \tilde d_{\tilde r}(\tilde x, \tilde y)=0\}$ and $\pi'(\tilde x):=\{\tilde y \in
\tilde X_{p,\tilde r'}: \tilde d_{\tilde r'}(\tilde x, \tilde y)=0\}.$

Let $X$ and $Y$ be metric spaces. Recall that a map $f:X\rightarrow Y$ is called an
\emph{isometry} if $f$ is distance-preserving and onto.

\begin{definition}\label{D1.4}A pretangent $\Omega_{p,\tilde
r}^{X}$ is tangent if $em':\Omega_{p,\tilde r}^{X}\rightarrow \Omega_{p,\tilde r'}^{X}$
is an isometry for every~$\Omega_{p, \tilde r'}^{X}.$\end{definition}

The following lemma is a direct corollary of Lemma 5 from \cite{DAK}.

\begin{lemma}\label{Lem1.6}
Let $(X,d,p)$ be a pointed metric space, $\mathbf{\mathfrak{B}}$ a countable subfamily
of $\tilde X$ and let $\tilde r=\{r_n\}_{n\in\mathbb N}$ be a normilizing sequence.
Suppose that $\tilde b$ and $\tilde p$ are mutually stable for every $\tilde
b=\{b_n\}_{n\in\mathbb N}\in\mathbf{\mathfrak{B}}.$ Then there is an infinite
subsequence $\tilde r'=\{r_{n_k}\}_{k\in\mathbb N}$ of $\tilde r$ such that the family
$$\mathbf{\mathfrak{B'}}:=\{\tilde b'=\{b_{n_k}\}_{k\in\mathbb N}: \tilde b \in
\mathbf{\mathfrak{B}}\}$$ is self-stable w.r.t. $\tilde r'.$
\end{lemma}

\section {Boundedness of pretangent spaces and local strong right porosity}
\hspace*{\parindent} Let us recall the definition of the right porosity. This
definition and an useful collection of facts related to the notion of porosity can be
found in \cite{Th}. Let $E$ be a subset of $\mathbb R^{+}=[0,\infty).$
\begin{definition}\label{D1}
The local right porosity of $E$ at 0 is the quantity
\begin{equation*}\label{L1}
p^{+}(E,0):=\limsup_{h\to 0^{+}}\frac{\lambda(E,0,h)}{h}
\end{equation*}
where $\lambda(E,0,h)$ is the length of the largest open subinterval of $(0,h)$ that
contains no point of $E$. The set $E$ is strongly porous on the right at 0 if $p^{+}(E,0)=1.$
\end{definition}

It was proved in \cite{DAK} that a bounded tangent space to $X$ at $p$
exists if and only if the distance set
\begin{equation*}\label{Sp}
S_{p}(X):=\{d(x,p): x\in X\}
\end{equation*} is strongly porous on the right at 0.

$\bullet$ It is therefore reasonable to ask for which pointed metric spaces $(X,d,p)$ all pretangent spaces $\Omega_{p,\tilde r}^{X}$ are bounded?

$\bullet$ Is there a modification of the local strong porosity describing the boundedness of all pretangent spaces $\Omega_{p,\tilde r}^{X}?$

Our first goal is to introduce a desired modification of porosity.

Let $\tilde \tau=\{\tau_n\}_{n\in\mathbb N}$ be a sequence of real numbers. We shall say
that $\tilde \tau$ is \emph{almost decreasing} if the inequality $\tau_{n+1}\le\tau_{n}$ holds
for sufficiently large $n.$ Write $\tilde E_{0}^{d}$ for the set of almost decreasing
sequences $\tilde \tau$ with $\mathop{\lim}\limits_{n\to\infty}\tau_{n}=0$ and having
$\tau_{n}\in E\setminus \{0\}$ for $n\in\mathbb N.$

Define $\tilde I_{E}$ to be the set of sequences $\{(a_n, b_n)\}_{n\in\mathbb
N}$ of open intervals $(a_n,b_n)\subseteq\mathbb R^{+}$ meeting the following conditions:

\bigskip
$\bullet$ \emph{each $(a_n, b_n)$ is a connected component of the set $Ext E=Int(\mathbb
R^{+}\setminus E),$ i.e., $(a_n,b_n)\cap E=\varnothing$ but for every
$(a,b)\supseteq(a_n,b_n)$ we have $$((a,b)\ne (a_n, b_n))\Rightarrow((a,b)\cap E \ne
\varnothing).$$}

 \emph{$\bullet$
$\mathop{\lim}\limits_{n\to\infty}a_{n}=0$ and
$\mathop{\lim}\limits_{n\to\infty}\frac{b_n-a_n}{b_n}=1.$}

Define also the weak equivalence $\asymp$ on the set of sequences of strictly positive
numbers as follows. Let $\tilde a=\{a_n\}_{n\in\mathbb N}$ and
$\tilde{\gamma}=\{\gamma_n\}_{n\in\mathbb N}.$ Then $\tilde a \asymp \tilde {\gamma}$ if
there are  constants $c_1, c_2
>0$ such that
\begin{equation*}\label{equiv1}
c_1 a_n < \gamma_n < c_2 a_n, \, n\in\mathbb N.
\end{equation*}

\begin{definition}\label{D2*}
Let 0 be an accumulation point of a set $E\subseteq\mathbb R^{+}$ and let $\tilde \tau
\in \tilde E_{0}^{d}.$ The set $E$ is $\tilde \tau$-strongly porous at 0 if there is a
sequence $\{(a_n, b_n)\}_{n\in\mathbb N}\in\tilde I_{E}$ such that
\begin{equation*}\label{equiv2}
\tilde\gamma \asymp \tilde a
\end{equation*}
where $\tilde a=\{a_n\}_{n\in\mathbb N}.$
\end{definition}

Let $E$ be a subset of $\mathbb R^{+}$ and let $0\in E.$
\begin{definition}\label{D2.3}
The set $E$ is $w$-strongly porous at 0 if for every sequence $\tilde
\tau=\{\tau_n\}_{n\in\mathbb N}\in\tilde E_{0}^{d}$ there is a subsequence $\tilde
\tau^{'}~=~\{\tau_{n_{k}}\}_{k\in\mathbb N}~\in~\tilde E_{0}^{d}$ for which the set $E$ is
$\tilde \tau^{'}$-strongly porous at~0.
\end{definition}
\begin{remark}\label{r2.4}
It is clear that $E\subseteq\mathbb R^{+}$ is $w$-strongly porous at 0 if 0 is an isolated point of $E$ and, on the other hand, if $E$ is $w$-strongly porous at 0 then $E$ is strongly porous at 0.
\end{remark}

The following theorem gives a boundedness criterion for pretangent spaces.
\begin{theorem}\label{Th2.4}
Let $(X, d, p)$ be a pointed metric space.
All pretangent spaces to $X$ at $p$ are bounded if and only if the set $S_{p}(X)$ is
$w$-strongly porous at $0.$
\end{theorem}

The proof of Theorem~\ref{Th2.4} is based on several auxiliary results. In the following proposition we consider the distance set $S_{p}(X)$ as a pointed metric space with the standard metric induced from $\mathbb R$ and the marked point 0.

\begin{proposition}\label{P2.5}
Let $(X,d,p)$ be a pointed metric space with the distance set $S_{p}(X).$ The following statements are equivalent.
\item[\rm(i)]\textit{All pretangent spaces to $X$ at $p$ are bounded.}
\item[\rm(ii)]\textit{All pretangent spaces to $S_{p}(X)$ at 0 are bounded.}
\end{proposition}
\begin{proof}
(ii) $\Rightarrow $ (i) Write
$\Omega_{0,\tilde r}^{S_{p}(X)}$ and $\tilde
S_{0, \tilde r} (X)$ for pretangent spaces to $S_{p}(X)$ at $0$ and, respectively, for the corresponding maximal self-stable families. Suppose that the inequality
$$\textrm{diam}\, \Omega_{0,\tilde r}^{S_{p}(X)}< \infty$$ holds for every $\Omega_{0,\tilde r}^{S_{p}(X)}.$ Let $\Omega_{p, \tilde
r}^{X}$ be an arbitrary pretangent space to $X$ with the corresponding  maximal
self-stable family $\tilde X_{p,\tilde r}$. Let $\tilde x, \tilde y \in \tilde X_{p, \tilde r}.$ The membership relations
$\tilde x, \tilde y \in \tilde X_{p, \tilde r}$ imply the existence of the finite limits
$$\lim_{n\to\infty}\frac{d(x_n, p)}{r_n}\qquad \mbox{and} \qquad \lim_{n\to\infty}\frac{d(y_n,p)}{r_n}.$$
As it was shown in (\cite{DAKM}, Proposition 2.2) the statement ``\emph{If $\tilde a$ and $ \tilde 0$ are mutually stable and $\tilde b$ and $\tilde 0$ are mutually stable, then $\tilde a$ and $\tilde b$ are mutually stable}'' holds for every normalizing sequence $\tilde r$ and every subspace $E$ of the metric space $\mathbb R^{+}$ with $0\in E$ and $\tilde a, \tilde b\in\tilde E.$

Consequently we obtain that $\{d(x_n,p)\}_{n\in\mathbb N},$ $\{d(y_n,p)\}_{n\in\mathbb
N}\in\tilde S_{0,\tilde r}(X).$ \linebreak Using the triangle inequality, we obtain $$\tilde
d_{\tilde r}(\tilde x, \tilde
y)=\lim_{n\to\infty}\frac{d(x_n,y_n)}{r_n}\le\lim_{n\to\infty}\frac{d(x_n,p)}{r_n}+\lim_{n\to\infty}\frac{d(y_n,p)}{r_n}
$$
\begin{equation*}
\le 2\sup_{\tilde z\in\tilde S_{0, \tilde r}(X)}\tilde d_{\tilde r}(\tilde 0, \tilde z)\le 2\, \textrm{diam}\,\Omega_{0, \tilde r}^{S_{p}(X)}.
\end{equation*}
Hence
\begin{equation*}
\textrm{diam}\,\Omega_{p, \tilde r}^{X}=\sup_{\tilde x, \tilde y\in\tilde X_{p,\tilde r}}\tilde d_{\tilde r}(\tilde x, \tilde y)\le 2\,\textrm{diam}\Omega_{0, \tilde r}^{S_{p}(X)}.
\end{equation*}
The boundedness of $\Omega_{p,\tilde r}^{X}$ follows.

(i) $\Rightarrow $ (ii) Suppose that all $\Omega_{p, \tilde r}^{X}$ are bounded but
there is an unbounded $\Omega_{0, \tilde r}^{S_{p}(X)}.$ Let $\tilde S_{0,\tilde
r}(X)$ be the maximal self-stable family corresponding to $\Omega_{0, \tilde
r}^{S_{p}(X)}.$ Since $\Omega_{0, \tilde r}^{S_{p}(X)}$ is unbounded we can find
a countable family of the sequences $\{d(p,b_{n}^{j})\}_{n\in\mathbb N}\in \tilde S_{0,
\tilde r}(X),$ $j\in\mathbb N,$ such that
\begin{equation}\label{L3.2}
\infty>\lim_{n\to\infty}\frac{d(p,b_{n}^{j})}{r_n}\ge j
\end{equation}
for every $j\in\mathbb N.$ By Lemma~\ref{Lem1.6} there is a subsequence $\tilde
r'=\{r_{n_k}\}_{k\in\mathbb N}$ such that the family of sequences
$\{b_{n_k}^{j}\}_{k\in\mathbb N,}$ $j\in\mathbb N,$ is self-stable w.r.t. $\tilde r'.$
Applying the Zorn Lemma we find a maximal self-stable family $\tilde X_{p, \tilde r'}$
such that $\{b_{n_k}^{j}\}_{k\in\mathbb N}\in\tilde X_{p,\tilde r'}$ for every $j.$
Inequalities \eqref{L3.2} imply that the pretangent space corresponding to $\tilde
X_{p,\tilde r'}$ is unbounded, contrary to the supposition.
\end{proof}

The next lemma was proved in (\cite{DB}, Corollary 2.4).

\begin{lemma}\label{l2.6} Let $E\subseteq\mathbb R^{+}$ and  let $\tilde\tau=\{\tau_{n}\}_{n\in\mathbb N}\in\tilde E_0^{d}.$ The set $E$ is
$\tilde\tau$-strongly porous if and only if there exists a
sequence $\{(a_n, b_n)\}_{n\in\mathbb N}\in\tilde I_{E}$ such that
$\mathop{\limsup}\limits_{n\to\infty}\frac{a_n}{\tau_n}<\infty$ and
$\tau_n\le a_n$ for sufficiently large $n.$
\end{lemma}
\begin{proposition}\label{P2.7}
Let $E\subseteq\mathbb R^{+}$  and let
$\tilde \tau=\{\tau_n\}_{n\in\mathbb N}\in\tilde E_{0}^{d}.$ The following statements
are equivalent.
\begin{enumerate}
\item[\rm(i)]  $E$ is $\tilde\tau$-strongly porous at 0.

\item[\rm(ii)] There is a constant $k\in (1, \infty)$ such that for every $K\in (k, \infty)$ there exists $N_{1}(K)\in \mathbb N$ such that
\begin{equation}\label{inters} (k\tau_n, K\tau_n)\cap E = \varnothing\end{equation} for $n
\ge N_1(K).$
\end{enumerate}
\end{proposition}

\begin{proof}(i) $\Rightarrow $ (ii)
Suppose that $E$ is $\tilde\tau$-strongly porous at 0. By Lemma~\ref{l2.6} there
is a sequence
\begin{equation}\label{int}
\{(a_n, b_n)\}_{n\in\mathbb N}\in\tilde I_{E}
\end{equation} such that $\mathop{\limsup}\limits_{n\to\infty}\frac{a_n}{\tau_n}<\infty$
and $\tau_n \le a_n$ for sufficiently large $n.$ Write
$$k=1+\mathop{\limsup}\limits_{n\to\infty}\frac{a_n}{\tau_n},$$ then $k\ge 2$ and there is
$N_0 \in \mathbb N$ such that
\begin{equation}\label{L3*}
\tau_n\le a_n< k\tau_n
\end{equation} for $n\ge N_0.$
Let $K\in (k, \infty).$ Membership \eqref{int} implies the equality
$\mathop{\lim}\limits_{n\to\infty}\frac{b_n}{a_n}=\infty.$ The last equality and
\eqref{L3*} show that there is $N_1\ge N_0$ such that  $$a_n< k\tau_n < K\tau_n \le
b_n$$ if $n\ge N_1.$ Hence the inclusion \begin{equation}\label{L3**} (k\tau_n, K
\tau_n)\subseteq(a_n,b_n) \end{equation} holds for $n\ge N_1.$ Since
\begin{equation}\label{L3***} E\cap(a_n, b_n)=\varnothing,\end{equation}\eqref{L3**} and
\eqref{L3***} imply \eqref{inters}. Thus (ii) follows from (i).

(ii) $\Rightarrow $ (i) Assume that statement $\textrm{(ii)}$ holds. Then for $K=2k$ there is $N_0\in\mathbb N$
such that $$(k\tau_n, 2k\tau_n)\cap E=\varnothing$$ for $n\ge N_0.$ Consequently, for
every $n\ge N_0,$ we can find a connected component $(a_n, b_n)$ of $Ext E$ meeting the
inclusion \begin{equation*}\label{inql} (k\tau_n, 2k\tau_n)\subseteq (a_n, b_n).
\end{equation*}
Define $(a_n, b_n):=(a_{N_0}, b_{N_0})$ for $n< N_0.$ Since, for $n\ge N_0,$ we have
$$\tau_n\in E,\, \tau_n < k\tau_n \, \, \mbox{and}\,\,(a_n, k\tau_n)\cap E=\varnothing,$$
the double inequality $\tau_n \le a_n< k\tau_n$ holds. Hence $\{\tau_n\}_{n\in\mathbb N}\asymp\{a_n\}_{n\in\mathbb N},$ i.e., to prove
$\textrm{(i)}$ it is sufficient to show that
 \begin{equation}\label{eq2.12}
 \{(a_n, b_n)\}_{n\in\mathbb N}\in \tilde
I_E.
\end{equation}
 All $(a_n,b_n)$ are connected components of $Ext E,$ so that \eqref{eq2.12}
holds if and only if
\begin{equation}\label{infty} \lim_{n\to\infty}\frac{b_n}{a_n}=\infty.
\end{equation} Let $K$ be an arbitrary point of $(k, \infty).$ Applying \eqref{inters} we can find $N_1(K)\in\mathbb N$ such that $$(k\tau_n, K\tau_n)\subseteq (a_n, b_n)$$ for $n\ge N_1 (K).$
Consequently, for such $n,$ we have
$$\frac{b_n}{a_n}\ge\frac{K\tau_n}{k\tau_n}=\frac{K}{k}.$$ Letting $K\to\infty$ we see that \eqref{infty} follows.
\end{proof}

 \noindent \emph{Proof of Theorem~\ref{Th2.4}.}
 The theorem is trivial if $p$ is an isolated point of $X,$ so that suppose $p$ is an accumulation point of $X.$
Taking into account Proposition~\ref{P2.5}, we can also assume that $X\subseteq\mathbb
R^{+}$ and $p=0.$

Let $X$ be $w$-strongly porous at $0$ and let $\Omega_{0, \tilde r}^{X}$ be an
arbitrary pretangent space to $X$ at $0$ with the corresponding maximal self-stable family $\tilde X_{0,\tilde r}.$  To prove that $\Omega_{0,\tilde r}^{X}$ is bounded it suffices to show that
\begin{equation}\label{eq5.3}
\sup_{\substack{\beta\in\Omega_{0, \tilde r}^{X} \\ \beta\ne\alpha}}\rho(\alpha, \beta)<\infty
\end{equation}
where $\alpha=\pi(\tilde 0)$ (see \eqref{eq1.5}). This inequality is vacuously true if $\Omega_{0, \tilde r}^{X}$ is one-point. In the case of card$\Omega_{0, \tilde r}^{X}\ge 2$ we can find  $c\in(0, \infty)$ and $\tilde \tau=\{\tau_n\}_{n\in\mathbb
N}\in\tilde X_{0,\tilde r}$ such that
\begin{equation}\label{eq2.15}\rho(\pi(\tilde\tau), \alpha)=\lim_{n\to\infty}\frac{\tau_n}{r_{n}}=c.\end{equation}  Let $\mathbb N^{i}$ be the set of all infinite subsets of $\mathbb N.$ Since $X$ is $w$-strongly porous at 0, there is $A\in\mathbb N^{i}$ such that $\tilde\tau'=\{\tau_n\}_{n\in A}$ is almost decreasing and $X$ is $\tilde\tau'$-strongly porous at 0. Note that the equivalence $\tilde x \asymp \tilde\tau$ holds for $\tilde x\in\tilde X_{0,\tilde r}$ if and only if $\pi(\tilde x)\ne\alpha.$ Using \eqref{eq2.15} we can write \eqref{eq5.3} in the equivalent form

\begin{equation}\label{eq2.16*}
\sup_{\substack{\tilde x\in\tilde X_{0,\tilde r} \\ \tilde x\asymp \tilde\tau}}\inf_{\substack{B\subseteq A \\ B\in\mathbb N^{i}}}\limsup_{\substack{n\to\infty \\ n\in B}}\frac{x_n}{\tau_n}<\infty.
\end{equation}
Let $\tilde x=\{x_n\}_{n\in\mathbb N}$ be an arbitrary element of $\tilde X_{0, \tilde r}$ for which $\pi(\tilde x)\ne\alpha.$ Then $\tilde x\asymp \tilde \tau$ holds. Moreover it is easy to find $B \subseteq A,$  $B\in\mathbb N^{i},$ such that $\{x_n\}_{n\in B}$ is almost decreasing. Since $X$ is $\tilde\tau'$-strongly porous at 0 and $\tilde x\asymp \tilde \tau,$ the set $X$ is also  $\{x_n\}_{n\in B}$-strongly porous at 0. Let $\{(a_n, b_n)\}_{n\in B}\in\tilde I_{X}$ be a sequence such that $\{\tau_{n}\}_{n\in B}\asymp\{a_n\}_{n\in B}\asymp\{x_{n}\}_{n\in B}.$
Lemma~\ref{l2.6}
implies that $$x_{n}\le a_n$$ for sufficiently large $n\in B$ and that
$$\limsup_{\substack{n\to\infty \\ n\in A}}\frac{a_n}{\tau_n}<\infty.$$ Thus
$$
\limsup_{\substack{n\to\infty \\ n\in B}}\frac{x_{n}}{\tau_n}\le\limsup_{\substack{n\to\infty \\ n\in B}}\frac{a_n}{\tau_n}\le\limsup_{\substack{n\to\infty \\ n\in A}}\frac{a_n}{\tau_n} <\infty.$$
Inequality~\ref{eq2.16*} follows.

Suppose now that all pretangent spaces to $X$ at 0 are bounded but the set $X$ is not
$w$-strongly porous at 0. By Definition~\ref{D2.3} there exists a decreasing sequence
$\tilde\tau=\{\tau_n\}_{n\in\mathbb N}$ such that $\tau_{n}\in X\setminus\{0\}$ for every $n\in\mathbb N,$ $\mathop{\lim}\limits_{n\to\infty}\tau_{n}=0$ and  $X$ is not
$\tilde\tau'$-strongly porous at 0 for every subsequence $\tilde\tau'$ of the sequence
$\tilde \tau.$ Since $X$ is not $\tilde\tau$-strongly porous, by Proposition~\ref{P2.7}
for every $k_{1}>1$ there is $K_1 \in (k_1, \infty)$ such that $(k_1 \tau_n, K_1 \tau_n)\cap
E\ne\varnothing$ for all $n$ belonging to an infinite set $A_{1}\subseteq\mathbb N.$
Using this fact we can find a convergent subsequence $$
\frac{\tilde x^{(1)}}{\tilde
\tau^{(1)}}:=\left\{\frac{x_{i}^{(1)}}{\tau_{i}}\right\}_{i\in A_{1}} \quad \mbox{such
that} \quad k_1 <\frac{x_{i}^{(1)}}{\tau_{i}}< K_{1}, \quad i\in A_{1},$$ $\tilde
\tau^{(1)}:=\{\tau_{i}\}_{i\in A_{1}},$ $\tilde x^{(1)}:=\{x_{i}^{(1)}\}_{i\in
A_{1}}, x_{i}^{(1)}\in X\setminus\{0\}.$ Let $k_2 > K_1 \vee 2.$ Since $X$ is not $\tilde\tau'$-strongly
porous, there are $K_2 > k_2$ and an infinite $A_2\subseteq A_1$ such that $(k_2 \tau_n,
K_2 \tau_n)\cap E\ne\varnothing$ for $n\in\mathbb N,$ so that we can construct a
convergent subsequence $$ \frac{\tilde x^{(2)}}{\tilde
\tau^{(2)}}:=\left\{\frac{x_{i}^{(2)}}{\tau_{i}}\right\}_{i\in A_{2}} \mbox{such
that} \quad k_2 <\frac{x_{i}^{(2)}}{\tau_{i}}< K_{2}, \quad x_{i}^{(2)}\in X\setminus\{0\},\, i\in A_{2}$$ where $\tilde
\tau^{(2)}=\{\tau_{i}^{(2)}\}_{i\in A_{2}}$ is a subsequence of $\tilde\tau^{(1)}$ and
$\tilde x^{(2)}:=\{x_{i}^{(2)}\}_{i\in A_{2}}.$ Repeating this procedure
we see that, for every $j\in\mathbb N,$ there are some sequences $\tilde
x^{(j+1)}=\{x_{i}^{(j+1)}\}_{i\in A_{j+1}}, x_{i}^{j+1}\in X\setminus\{0\}$ and $\tilde
\tau^{(j+1)}=\{\tau_{i}\}_{i\in A_{j+1}},$ $A_{j+1}$ is infinite subset of $A_{j}\subseteq\mathbb N,$ such that $$k_{j}\vee
k_{j+1}<\frac{x_{i}^{(j+1)}}{\tau_{i}}< K_{j+1}, $$ for $i\in A_{j+1}$ and
$$\frac{\tilde x^{(j+1)}}{\tilde\tau^{(i+1)}}:=\left \{\frac{x_{i}^{(j+1)}}{\tau_{i}} \right\}_{i\in
A_{j+1}}$$ is convergent. To complete the proof, it suffices to make use of Cantor's diagonal argument.

Let
$B:=\{n_1, ..., n_{j},...\}$ be an infinite subset of $\mathbb N$ such that $n_{j}\in
A_{j}$ for every $j\in\mathbb N.$  Let us define the subsequences $\tilde
y^{(j)}=\{y_{k}^{j}\}_{k\in B}$ by the rule
\begin{equation*}
y_{k}^{(j)}:=\begin{cases}
       0 & \mbox{if} \quad k \in A_{j}\setminus B\\
       x_{k}^{(j)}& \mbox{if} \quad k\in A\cap A_{j}. \\
         \end{cases}
\end{equation*}
Then the sequence $\left \{\frac{y_{k}^{(j)}}{\tau_{k}} \right\}_{k\in B}$ is convergent
and
\begin{equation}\label{eq5.6}
j\le\lim_{\substack{k\to\infty \\ k\in B}}\frac{y_{k}^{(j)}}{\tau_{k}}
\end{equation} for every $j\in\mathbb N.$ Since all $\{y_{k}^{(j)}\}_{k\in B}$ are
mutually stable w.r.t. the norma\-lizing sequence
$\tilde\tau'=\{\tau_k\}_{k\in B},$ there is a maximal self-stable $\tilde X_{0, \tilde
r'}$ such that $\{y_{k}^{(j)}\}_{k\in B}\in\tilde X_{0, \tilde r'}$ for $j\in\mathbb N.$
Inequality \eqref{eq5.6} shows that the corresponding pretangent space $\Omega_{0,
\tilde r'}^{X}$ is unbounded, contrary to the supposition. $\qquad\quad\,\square$

\bigskip

\begin{corollary}
Let $(X,d,p)$ be a pointed metric space. If all pretangent spaces $\Omega_{p, \tilde r}^{X}$ are bounded, then at least one from these pretangent spaces is tangent.
\end{corollary}
\begin{proof}
Suppose that all $\Omega_{p, \tilde r}^{X}$ are bounded, then, by Theorem~\ref{Th2.4}, the set $S_{p}(X)$ is $w$-strongly porous at 0. Consequently $S_{p}(X)$ is strongly porous on the right at 0 (see Remark~\ref{r2.4}). As was noted above, $S_{p}(X)$ is strongly porous on the right at 0 if and only if there is a bounded tangent space $\Omega_{p, \tilde r}^{X}.$
\end{proof}
We shall say that a set $E\subseteq\mathbb{R}^{+}$ is \emph{completely strongly porous} at 0 if $E$ is $\tilde\tau$-strongly porous at 0 for every $\tilde\tau\in \tilde E_{0}^{d}.$ Some properties of completely strongly porous sets $E\subseteq\mathbb R^{+}$ are described in \cite{DB}.

The following example shows that there exist $w$-strongly porous at 0 subsets of $\mathbb
R^{+}$ which are not completely strongly porous at 0.
\begin{example}\label{ex2.8}
Let $\tau_1 = 1$ and $\tau_{n+1}=2^{-n^{2}}\tau_n$ for every $n\in\mathbb N.$
Let $\mathbb N_{1}, \mathbb N_{2}, ..., \mathbb N_{k},...$ be an infinite partition of $\mathbb N,$ $$\bigcup_{k=1}^{\infty}\mathbb N_{k}=\mathbb N, \quad \mathbb N_{i}\cap\mathbb N_{j}=\varnothing \quad \mbox{for}\quad i\ne j,$$
such that $$\nu(1)<\nu(2)<...<\nu(k)<...$$ where $$\nu(k):=\min_{n\in\mathbb N_{k}}n.$$
For every $n\in\mathbb N$ define $\tau_{n}^{*}$ as $2^{-\nu(m(n))}\tau_{n}$ where $m(n)$ is the index for which $n\in\mathbb N_{m(n)}.$ Write $E_{1}:=\{\tau_n: n\in\mathbb N\},$ $E_{1}^{*}:=\{\tau_{n}^{*}: n\in\mathbb N\}$ and $$E:=E_{1}\cup E_{1}^{*}\cup\{0\},$$
here $E_{1}$ and $E_{1}^{*}$ are the ranges of the sequences $\{\tau_n\}_{n\in\mathbb N}$ and $\{\tau_{n}^{*}\}_{n\in\mathbb N}$ respectively. Using Lemma~\ref{l2.6} we can show that $E$ is not $\tilde\tau$-strongly porous with $\tilde\tau=\{\tau_n\}_{n\in\mathbb N}$ define as above, so that $E$ is not completely strongly porous.

Let us show that $E$ is $w$-strongly porous at 0. Note that, for every $\tilde y=\{y_n\}_{n\in\mathbb N}\in \tilde E_{0}^{d},$
there are three possibilities:
\newline (i) $y_n \in E_{1}$ holds for an infinite number of subscripts $n;$
\newline (ii) there is $k\in\mathbb N$ such that $$\textrm{card}(\{y_n: n\in\mathbb N\}\cap\{\tau_{n}^{*}: n\in\mathbb N_{k}\})=\infty;$$
\newline (iii) there is an infinite strictly increasing sequence $\{k_i\}_{i\in\mathbb N}$ such that $$\{y_n: n\in\mathbb N\}\cap\{\tau_{n}^{*}: n\in\mathbb N_{k_i}\}\ne\varnothing.$$

It follows directly from the definitions that
\begin{equation}\label{eqv2.13}
n\ge\nu(k)\ge k
\end{equation}
for every $k\in\mathbb N$ and every $n\in\mathbb N_k.$ This double inequality implies that $n\ge\nu(m(n)).$ Using the last inequality and definitions of $\tau_n$ and $\tau_{n}^{*}$ we obtain \begin{equation}\label{eqv2.14}
\tau_{n+1}=2^{-n^{2}}\tau_{n}\le 2^{-n}\tau_{n}\le 2^{-\nu(m(n))}\tau_{n}=\tau_{n}^{*}<\tau_{n}.
\end{equation}
In particular, \eqref{eqv2.14} implies that $\tau_{n}^{*}=\tau_{n}$ is possible only for $n=1$ and that $$\tau_{n+1}^{*}<\tau_{n+1}<\tau_{n}^{*}<\tau_{n}$$ holds for every $n\ge 2.$ Moreover we obtain from \eqref{eqv2.14} that
\begin{equation}\label{eqv2.15}
\lim_{n\to\infty}\frac{\tau_{n+1}}{\tau_{n}^{*}}\le\lim_{n\to\infty}\frac{2^{-n^{2}}}{2^{-n}}=0.
\end{equation}
Consequently $\{(\tau_{n+1}, \tau_{n}^{*})\}_{n\in\mathbb N}\in\tilde I_{E}.$ Using the last membership we can show that for $\tilde y$ satisfying (i), there is $\tilde y'$ such that $E$ is $\tilde y'$-strongly porous at 0. If $\tilde y$ meets condition (ii), then using \eqref{eqv2.15} and the equality $\tau_{n}^{*}=2^{-\nu(k)}\tau_{n}, \, n\in\mathbb N_{k},$ we can also find the desired $\tilde y'.$ Finally, if (iii) holds, then $\tilde y'$ can be constructed with the use of the relation $$\lim_{k\to\infty}2^{-\nu(k)}=0,$$ wich follows from the second inequality in \eqref{eqv2.13}. We leave the details of the constructions of $\tilde y'$ to the reader.
\end{example}
\begin{remark}\label{r2.11}
Considering $E$ from Example~\ref{ex2.8} as a pointed metric space with a marked point 0 we can show that the inequality card$(\Omega_{0,\tilde r}^{E})\le 3$ holds for every $\Omega_{0,\tilde r}^{E}.$  On the other hand, if $(X,d,p)$ is a pointed metric space such that card$(\Omega_{p,\tilde r}^{E})\le 2$ holds for every $\Omega_{p,\tilde r}^{E},$ then the distance set $\{d(x,p): x\in X\}$ is completely strongly porous at 0.
\end{remark}


\medskip

\textbf{Viktoriia Bilet}

Institute of Applied Mathematics and Mechanics of NASU, R. Luxemburg str. 74, Donetsk 83114, Ukraine

\textbf{E-mail:} biletvictoriya@mail.ru

\bigskip

\textbf{Oleksiy Dovgoshey}

Institute of Applied Mathematics and Mechanics of NASU, R. Luxemburg str. 74, Donetsk 83114, Ukraine

\textbf{E-mail:} aleksdov@mail.ru
\end{document}